\documentclass[authoryear]{elsarticle}

\usepackage[a4paper, total={5.5in, 8in}]{geometry}
\usepackage{booktabs}
\usepackage{float}
\usepackage{amsmath}
\usepackage{amsthm}
\usepackage[colorlinks]{hyperref}

\usepackage{amsmath,amssymb,amsfonts,mathrsfs,hyperref,microtype, amsthm}
\usepackage{graphicx}
\usepackage{graphicx,mathabx}
\graphicspath{ {images/} }
\numberwithin{equation}{section}

\newtheorem{theorem}{Theorem}[section]
\newtheorem{lemma}{Lemma}[section]
\newtheorem{corollary}{Corollary}[section]
\newtheorem*{remark}{Remark}
\newcommand{\BS}{\boldsymbol}

\newcommand{\Rmnum}[1]{\expandafter\@slowromancap\romannumeral #1@}
\newcommand{\trp}{{\sf T}}
\journal{}
\makeatletter
\def\ps@pprintTitle{%
   \let\@oddhead\@empty
   \let\@evenhead\@empty
   \def\@oddfoot{\reset@font\hfil\thepage\hfil}
   \let\@evenfoot\@oddfoot
}

\begin{document}
\begin{frontmatter}
\author[a]{ Osama Idais\corref{cor1}}
\cortext[cor1]{Corresponding author}
\ead{osama.idais@ovgu.de}
\address[a]{\small Institute for Mathematical Stochastics, Otto-von-Guericke University Magdeburg, \\ \small PF 4120, D-39016 Magdeburg, Germany}

 \title {A note on locally optimal designs for generalized linear models with restricted support}

\begin{abstract}
Optimal designs for generalized linear models require a prior knowledge of the regression parameters. At certain values of the parameters we propose particular assumptions which allow to derive a locally optimal design for a model without intercept from a locally optimal design for the corresponding model with intercept and vice versa.  Applications to Poisson and logistic models and Extensions to nonlinear models are provided. 
\end{abstract}
\begin{keyword}
approximate design\sep information matrix\sep  model without intercept\sep optimal design\sep  saturated design.

\end{keyword}

\end{frontmatter}

\section{Introduction}
\label{}

The generalized linear model, GLM, is a generalization of the ordinary linear regression which allows continuous or discrete observations from  one-parameter exponential family distributions to be combined with explanatory variables (factors) via proper link functions (\citet{10.2307/2344614}). In GLM framework  logistic, probit, Poisson and  gamma models are included besides others (\citet{mccullagh1989generalized} and  \citet{dobson2018introduction}). Therefore, wide applications deal with GLMs such as social and educational sciences, clinical trials, insurance and  industry. \par

The information matrix for a GLM depends on the model parameters. Locally optimal designs under GLMs are derived at a certain value of the parameters (\citet{khuri2006}, \citet{atkinson2015designs}).  A possible procedure to overcome the complexity in deriving a locally optimal design for GLMs without intercept is to make use of an available locally optimal design for GLMs with intercept and vice versa. This procedure was suggested in  \cite{10.2307/24307120} to investigate the relation between optimal designs for mixture and for component amount models. Their result was extended under linear models in \cite{Li2005} to derive a D-optimal design for a non-intercept linear model from that for a linear model with intercept.  In contrast,  \cite{ZHANG2013196} provided specific conditions to derive D- and A-optimals for component amount models (with intercept) from  analogous optimal designs  for the corresponding mixture models (without intercept). In this paper we generalize their approaches for GLMs under D- and A-criteria and we introduce a more transparent proof based on The General Equivalence Theorem. 
This paper is organized as follows. In Section 2, the models and design optimality criteria are introduced.  In Section 3, we present the main results followed by applications to Poisson and logistic models in Section 4. Further extensions are given in Section 5.

 \section{Models and designs}
 Let $Y(\BS{x}_1), ... ,Y(\BS{x}_n)$ be independent response variables at $n$ experimental conditions $\BS{x}_1,\dots, \BS{x}_n$ which come from an experimental region   ${\cal X}\subseteq  \mathbb{R}^{\nu}, \nu\ge1$, i.e., $\BS{x}_i\in \mathcal{X}, i=1,\dots,n$.  Under generalized linear models with  the vector of model parameters $\BS{\beta}\in \mathbb{R}^{p}$  each observation $Y(\BS{x}_i)$ belongs to a  one-parameter exponential family distribution with expected mean $E(Y(\BS{x}_i),\BS{\beta})=\mu(\BS{x}_i,\BS{\beta})$ and variance  $\mathrm{var}(Y(\BS{x}_i),\BS{\beta})=a(\phi)V(\mu(\BS{x}_i,\BS{\beta}))$ where  $V(\mu(\BS{x}_i,\BS{\beta}))$ is a mean-variance function and $\phi$  is a dispersion parameter (see \citet{mccullagh1989generalized}, Section 2.2.2). 
  Let $\BS{f}(\BS{x}):{\cal X}\rightarrow \mathbb{R}^{p}$ be a $p$-dimensional regression function written as  $\BS{f}(\BS{x})=(f_{1}(\BS{x}),\dots, f_{p}(\BS{x}))^\trp$. To assure estimability of the parameters the components $f_{1}(\BS{x}),\dots, f_{p}(\BS{x})$ are assumed to be real-valued continuous linearly independent functions on $\mathcal{X}$. The expected mean $\mu(\BS{x}_i,\BS{\beta})$ is related to a linear predictor  $\eta(\BS{x}_i,\BS{\beta})=\BS{f}^\trp(\BS{x}_i)\BS{\beta}$ via a one-to-one and differentiable 
 link function $g$, i.e., $\eta(\BS{x}_i,\BS{\beta})=g(\mu(\BS{x}_i,\BS{\beta}))$, $i=1,\dots,n$.  We can define the  intensity function for each point $\BS{x}_i \in \mathcal{X}$ as
\begin{equation}
 u(\BS{x}_i,\BS{\beta})=\Big(a(\phi)V\bigl(g^{-1}\big(\BS{f}^\trp(\BS{x}_i)\BS{\beta}\big)\bigr)\Big)^{-1} \Big(g^{\prime}\bigl(g^{-1}\big(\BS{f}^\trp(\BS{x}_i)\BS{\beta}\big) \bigr)\Big)^{-2}\,\,\,(1\le i \le n) \label{eq2.3}
  \end{equation}
 which is positive and depends on the value of linear predictor $\BS{f}^\trp(\BS{x}_i)\BS{\beta}$ (\citet{atkinson2015designs}). 
The Fisher information matrix for a GLM  can be given in the form $\BS{M}(\BS{x}_i,\BS{\beta})=u(\BS{x}_i,\BS{\beta})\,\BS{f}(\BS{x}_i)\,\BS{f}^\trp(\BS{x}_i)$ for all  $i=1,\dots, n$ (see \citet{fedorov2013optimal}, Subsection 1.3.2). Define the function  $\BS{f}_{\BS{\beta}}(\BS{x}_i)=u^{\frac{1}{2}}(\BS{x}_i,\BS{\beta})\BS{f}(\BS{x}_i)$ then the Fisher information matrix may rewrite as $\BS{M}(\BS{x}_i,\BS{\beta})=\BS{f}_{\BS{\beta}}(\BS{x}_i)\BS{f}_{\BS{\beta}}^{\trp}(\BS{x}_i)$ for each $i=1,\dots,n$. The latter form is appropriate for other nonlinear models and will appear frequently in the paper. For the whole experimental conditions $\BS{x}_1, \dots, \BS{x}_n$ the Fisher information matrix can be obtained by  $
\BS{M}(\BS{x}_1,\dots,\BS{x}_n,\BS{\beta})=\sum_{i=1}^{n}\BS{M}(\BS{x}_i,\BS{\beta})$.\par

In this article, we focus on approximate designs $\xi$ defined on the experimental region ${\cal X}$ with finite and mutually distinct support points  $\BS{x}_1, \BS{x}_2,\dots,\BS{x}_r$ and the corresponding weights  $\omega_1, \omega_2,\dots, \omega_r>0$ such that  $\sum_{i=1}^{r} \omega_i=1$ ( see \cite{silvey1980optimal}, p.15).
The set ${\rm supp}(\xi)=\{\BS{x}_1,\BS{x}_2, \dots,\BS{x}_r\}$ is called the support of $\xi$. The information matrix of a design $\xi$  at a parameter point $\BS{\beta}$ is defined by
\begin{eqnarray}
\BS{M}(\xi, \BS{\beta})=\int_{\mathcal{X}} \BS{M}(\BS{x}, \BS{\beta})\, \xi({\rm d}\BS{x})=\sum_{i=1}^{r}\omega_i \BS{M}(\BS{x}_i, \BS{\beta}).\label{eq2-3}
\end{eqnarray}
Optimal designs derived under specific optimality criteria.  Throughout, we restrict to the common D- and A-criteria. Denote by ``$\det$'' and ``${\rm tr}$'' the determinant and the trace of a matrix, respectively.   A design $\xi^*$ is called  locally D-optimal (at $\BS{\beta}$) if it minimizes $\det\bigl(\BS{M}^{-1}(\xi, \BS{\beta})\bigr)$  over all designs $\xi$ whose information matrix (at $\BS{\beta}$) is nonsingular. Similarly, a design $\xi^*$ is called locally A-optimal (at $\BS{\beta}$) if it minimizes   ${\rm tr}\bigl(\BS{M}^{-1}(\xi, \BS{\beta})\bigr)$  over all designs $\xi$ whose information matrix (at $\BS{\beta}$) is nonsingular. The General Equivalence Theorem can be used to investigate the optimality of a design with respect to  D-criterion  and  A-criterion (see \cite{silvey1980optimal}, p.40, p.48 and  p.54).  Let $\BS{\beta}$ be a given parameter point and let $\xi^*$ be a design with nonsingular information matrix $\BS{M}(\xi^*, \BS{\beta})$. The design $\xi^*$ is locally D-optimal (at $\BS{\beta}$) if and only if
\begin{equation}
u(\BS{x},\BS{\beta})\BS{f}^\trp(\BS{x})\BS{M}^{-1}(\xi^*, \BS{\beta}) \BS{f}(\BS{x})\leq p \,\,\mbox{for all } \BS{x} \in \mathcal{X}. \label{Equiv-D}
\end{equation}
The design $\xi^*$ is locally A-optimal (at $\BS{\beta}$) if and only if 
\begin{equation}
u(\BS{x},\BS{\beta})\BS{f}^\trp(\BS{x})\BS{M}^{-2}(\xi^*, \BS{\beta})\BS{f}(\BS{x})\leq 
\mathrm{tr}\bigl(\BS{M}^{-1}(\xi^*, \BS{\beta})\bigr) \,\, \mbox{for all } 
\BS  x \in \mathcal{X}. \label{Equiv-A} 
\end{equation}

\begin{remark} \label{rem2.3.1.}
The maximum of  inequality (\ref{Equiv-D}) or (\ref{Equiv-A}) achieves at the support points of any D- or A-optimal deigns, respectively. The left hand side of each inequality is called  the sensitivity function.
\end{remark}

\section{Main results}
In the following we distinguish between the model with an explicit intercept $\mathcal{M}$, say and  the corresponding model without an explicit intercept $\widetilde{\mathcal{M}}$, say. We modify our notations and thus these  models; $\widetilde{\mathcal{M}}$ and  $\mathcal{M}$ are (with out loss of generality) characterized in the following. 
 \begin{equation*}
 \widetilde{\mathcal{M}}\,:\,\,\, \Tilde{\eta}=\BS{f}^\trp(\BS{x})\BS{\Tilde{\beta}}\,\, \mbox{ where }\,\, \,\, \BS{x} \in\mathcal{\widetilde{X}} 
\end{equation*}
and  $\BS{\Tilde{\beta}}=(\beta_1,\dots,\beta_\nu)^\trp$. Denote the intensity function by $\Tilde{u}(\BS{x},\BS{\Tilde{\beta}})$ and let $\Tilde{u}_{0}=\Tilde{u}(\BS{0},\BS{\Tilde{\beta}})$. Here we assume there is no constant (intercept) term explicitly involved in the present model, i.e., none of the regression
 components of the   $\nu$ real-valued function $\BS{f}(\BS{x})$ is constant equal to $1$. 
Denote  $\BS{f}_{\BS{\Tilde{\beta}}}(\BS{x})= \Tilde{u}^{\frac{1}{2}}(\BS{x},\BS{\Tilde{\beta}})\BS{f}(\BS{x})=({f}^{(1)}_{\Tilde{\BS{\beta}}},\dots,{f}^{(\nu)}_{\Tilde{\BS{\beta}}})^{\trp}$ and thus the information matrix of $\xi$ on $\mathcal{\widetilde{X}}$  under model $ \widetilde{\mathcal{M}}$  is written as  
 \begin{equation*}
 \BS{\Tilde{M}}\big(\xi,\BS{\Tilde{\beta}}\big)=\int_{\mathcal{\widetilde{X}}} \BS{f}_{\BS{\Tilde{\beta}}}(\BS{x}) \BS{f}_{\BS{\Tilde{\beta}}}^\trp(\BS{x})\, \xi(\mathrm{d}\BS{x}).
 \end{equation*}

The corresponding model  $\mathcal{M}$ is defined by including the constant $1$ and  the intercept parameter $\beta_0$ into the linear predictor of the generalized linear model as in the following.
\begin{equation*}
 \mathcal{M}\,:\,\,\, \eta=\big(1,\BS{f}^\trp(\BS{x})\big)\BS{\beta}=\beta_0+\BS{f}^\trp(\BS{x})\BS{\Tilde{\beta}}\,\,\mbox{ where }\,\, \BS{x} \in\mathcal{X}
\end{equation*}
and $\BS{\beta}=(\beta_0,\BS{\Tilde{\beta}}^\trp)^\trp$. Denote the intensity function by $u(\BS{x},\BS{\beta})$ and let $u_{0}=u(\BS{0},\BS{\beta})$. 
Denote the function $\BS{f}_{\BS{\beta}}(\BS{x})=u^{\frac{1}{2}}(\BS{x},\BS{\beta})\BS{f}(\BS{x})=(f^{(1)}_{\BS{\beta}},\dots,f^{(\nu)}_{\BS{\beta}})^{\trp}$.  So we can write  $u^{\frac{1}{2}}(\BS{x},\BS{\beta})\big(1,\BS{f}^\trp(\BS{x})\big)^\trp=\big(u^{\frac{1}{2}}(\BS{x},\BS{\beta}),\BS{f}^{\trp}_{\BS{\beta}}(\BS{x}) \big)^\trp$. Define $\Xi_{0}$ to be the set of all designs on $\mathcal{X}$ for model $\mathcal{M}$ such that $ \BS{0}\in \mathrm{supp}(\xi)$ and there exist a constant vector $\BS{c}$ such that $\BS{c}^\trp\BS{f}(\BS{x})=1$ for all $\BS{x}\in \mathrm{supp}(\xi)\setminus\{\BS{0}\}$, i.e., 
\[
\Xi_{0}=\left\{\xi: \xi \mbox{ on } \mathcal{X} \mbox{  with }\BS{0}\in \mathrm{supp}(\xi) \mbox{ and }\exists\, \BS{c}\in \mathbb{R}^{\nu} \ni \BS{c}^\trp\BS{f}(\BS{x})=1\, \forall \BS{x}\in \mathrm{supp}(\xi)\setminus\{\BS{0}\} \right\}.
\]
Then the information matrix of $\xi\in\Xi_{0}$ under model $ \mathcal{M}$  reads as 
\begin{equation*} 
\BS{M}(\xi,\BS{\beta})=\int_{\mathcal{X}} \big(u^{\frac{1}{2}}(\BS{x},\BS{\beta}),\BS{f}_{\BS{\beta}}^\trp(\BS{x})\big)^\trp \big(u^{\frac{1}{2}}(\BS{x},\BS{\beta}),\BS{f}_{\BS{\beta}}^\trp(\BS{x})\big)\, \xi(\mathrm{d}\BS{x}).
\end{equation*}

In the following  we give sufficient conditions under which the locally D- resp. A-optimal design at a parameter point $\BS{\Tilde{\beta}}$  for model $\widetilde{\mathcal{M}}$ can be obtained from the locally D- resp. A-optimal design from $\Xi_0$ at a parameter point $\BS{\beta}=(\beta_0,\BS{\Tilde{\beta}}^\trp)^\trp$  for the corresponding model $\mathcal{M}$  by simply removing the origin point  from its support points and renormalizing the weights of the remaining support points and vice versa. To this end,  for a design $\xi\in \Xi_{0}$  define $\xi_{-\BS{0}}$ on $\mathcal{\widetilde{X}}\subseteq\mathcal{X}$ to be the  conditional measure of $\xi$ given $\BS{x}\neq \BS{0}$. So we get  $\mathrm{supp}(\xi)=\mathrm{supp}(\xi_{-\BS{0}})\cup \{\BS{0}\}$. Let $\xi_{\BS{0}}$ denotes the one point design supported by the origin point $\BS{0}$, then we can write $ \xi=\omega\,\xi_{\BS{0}}+(1-\omega)\,\xi_{-\BS{0}}$.  Assume that  for a given parameter point $\BS{\beta}=(\beta_0,\BS{\Tilde{\beta}}^\trp)^\trp$ we have  $u(\BS{x},\BS{\beta})=\Tilde{u}(\BS{x},\BS{\Tilde{\beta}})$ which yields  $\BS{f}_{\BS{\beta}}(\BS{x})=\BS{f}_{\BS{\Tilde{\beta}}}(\BS{x})$ and  $\BS{M}(\xi,\BS{\beta})=\BS{M}(\xi,\BS{\Tilde{\beta}})$ with $u_{0}=\Tilde{u}_{0}$. In particular, let  $\BS{f}_{\BS{\Tilde{\beta}}}(\BS{0})=\BS{0}$ then we find
\begin{equation*}
\BS{M}(\xi,\BS{\Tilde{\beta}})=\,\left(\begin{array}{cc} m_{1,1}(\xi,\BS{\Tilde{\beta}}) & (1-\omega)\,\BS{\Tilde{m}}^\trp(\xi_{-\BS{0}},\BS{\Tilde{\beta}})\\[0.3cm]    (1-\omega)\, \BS{\Tilde{m}}(\xi_{-\BS{0}},\BS{\Tilde{\beta}})& (1-\omega)\, \BS{\Tilde{M}}(\xi_{-\BS{0}},\BS{\Tilde{\beta}})\end{array} \right),
\end{equation*} 
where
\begin{align*}
&m_{1,1}(\xi,\BS{\Tilde{\beta}})=\int_{\mathcal{X}} \Tilde{u}(\BS{x},\BS{\Tilde{\beta}})\,\xi (\mathrm{d}\BS{x} ),\,\,
\BS{\Tilde{m}}(\xi_{-\BS{0}},\BS{\Tilde{\beta}})=\int_{\mathcal{\widetilde{X}}} \Tilde{u}^{\frac{1}{2}}(\BS{x},\BS{\Tilde{\beta}})\BS{f}_{\BS{\Tilde{\beta}}}(\BS{x}) \,\xi_{-\BS{0}} (\mathrm{d}\BS{x} )\,\, and \\
&\BS{\Tilde{M}}(\xi_{-\BS{0}},\BS{\Tilde{\beta}})= \int_{\mathcal{\widetilde{X}}} \BS{f}_{\BS{\Tilde{\beta}}}(\BS{x})\BS{f}_{\BS{\Tilde{\beta}}}^\trp(\BS{x}) \,\xi_{-\BS{0}} (\mathrm{d}\BS{x}).
\end{align*}
Note that the submatrix $\BS{\Tilde{M}}(\xi_{-\BS{0}},\BS{\Tilde{\beta}})$ is the information matrix  of $\xi_{-\BS{0}}$ for model  $\widetilde{\mathcal{M}}$. Furthermore,  $m_{1,1}(\xi,\BS{\Tilde{\beta}})=\omega \Tilde{u}_{0}+\widetilde{m}^{\circ}(\xi_{-\BS{0}},\BS{\Tilde{\beta}})$ where $\widetilde{m}^{\circ}(\xi_{-\BS{0}},\BS{\Tilde{\beta}})=\int_{\mathcal{\widetilde{X}}}\Tilde{u}(\BS{x},\BS{\Tilde{\beta}})\,\xi_{-\BS{0}} (\mathrm{d}\BS{x} )$. Since there exist a constant vector $\BS{c}$ such that $\BS{c}^\trp\BS{f}(\BS{x})=1$ for all $\BS{x}\in \mathrm{supp}(\xi)\setminus\{\BS{0}\}$, it is straightforward to verify the following  
\begin{align*}
&\BS{c}^\trp\BS{\Tilde{m}}(\xi_{-\BS{0}},\BS{\Tilde{\beta}})=\widetilde{m}^{\circ}(\xi_{-\BS{0}},\BS{\Tilde{\beta}})\,\,\mbox{ and  } \BS{\Tilde{M}}^{-1}(\xi_{-\BS{0}},\BS{\Tilde{\beta}})\BS{\Tilde{m}}(\xi_{-\BS{0}},\BS{\Tilde{\beta}})=\BS{c}\,\, \mbox{ thus }\\
&\BS{\Tilde{m}}^\trp(\xi_{-\BS{0}},\BS{\Tilde{\beta}})  \BS{\Tilde{M}}^{-1}(\xi_{-\BS{0}},\BS{\Tilde{\beta}})\BS{\Tilde{m}}(\xi_{-\BS{0}},\BS{\Tilde{\beta}})=\widetilde{m}^{\circ}(\xi_{-\BS{0}},\BS{\Tilde{\beta}}).
\end{align*}
As a result we get
 \begin{equation}
\BS{M}^{-1}(\xi,\BS{\Tilde{\beta}})=\left(\begin{array}{cc} \frac{1}{\omega\,\Tilde{u}_{0}}  & -\frac{\BS{c}^\trp}{\omega\,\Tilde{u}_{0}}\\ [0.3cm]   -\frac{\BS{c}}{\omega\,\Tilde{u}_{0}}& \frac{1}{1-\omega}\,  \BS{\Tilde{M}}^{-1}(\xi_{-\BS{0}},\BS{\Tilde{\beta}})+\frac{\BS{c}\BS{c}^\trp}{\omega\,\Tilde{u}_{0}}\end{array} \right). \label{eq3.15}
\end{equation}

\begin{lemma} \label{lemm3.3-1}
Consider design  $\xi^*\in \Xi_{0}$ for model $\mathcal{M}$.  Let a parameter point $\BS{\beta}=(\beta_0,\BS{\Tilde{\beta}}^\trp)^\trp$ be given such that $u(\BS{x},\BS{\beta})=\Tilde{u}(\BS{x},\BS{\Tilde{\beta}})$ for all $\BS{x}\in \mathcal{\widetilde{X}}\subseteq \mathcal{X}$ and $\BS{f}_{\BS{\Tilde{\beta}}}(\BS{0})=\BS{0}$. Then the design $\xi^*$ is  locally D-optimal (at $\BS{\beta}$) if  it assigns weight $\omega=(\nu+1)^{-1}$ to the origin $\BS{0}$.
\end{lemma}
\begin{proof} 
Under the assumptions given in the lemma we obtain $\BS{M}^{-1}(\xi^*,\BS{\Tilde{\beta}})$ from (\ref{eq3.15}).  Then  the sensitivity function obtained from condition (\ref{Equiv-D}) of The Equivalence Theorem is given by 
\begin{align*}
\psi(\BS{x},\xi^*)&=u(\BS{x},\BS{\Tilde{\beta}})\big(1,\BS{f}^\trp(\BS{x})\big)\BS{M}^{-1}(\xi^*,\BS{\Tilde{\beta}})\big(1,\BS{f}^\trp(\BS{x})\big)^\trp\\
&=u(\BS{x},\BS{\Tilde{\beta}})\Big(\BS{f}^\trp(\BS{x})\bigl(  \frac{1}{1-\omega}\,  \BS{\Tilde{M}}^{-1}(\xi_{-\BS{0}}^*,\BS{\Tilde{\beta}})+\frac{\BS{c}\BS{c}^\trp}{\omega\,\Tilde{u}_{0}} \bigr)\BS{f}(\BS{x})-2 \frac{\BS{c}^\trp}{\omega\,\Tilde{u}_{0}}\BS{f}(\BS{x})+\big(\omega \Tilde{u}_{0}\big)^{-1}\Big)\\
&=\BS{f}_{\BS{\Tilde{\beta}}}^\trp(\BS{x})\bigl(  \frac{1}{1-\omega}\,  \BS{\Tilde{M}}^{-1}(\xi_{-\BS{0}}^*,\BS{\Tilde{\beta}})+\frac{\BS{c}\BS{c}^\trp}{\omega\,\Tilde{u}_{0}} \bigr)\BS{f}_{\BS{\Tilde{\beta}}}(\BS{x})-2 \frac{\BS{c}^\trp}{\omega\,\Tilde{u}_{0}}\BS{f}_{\BS{\Tilde{\beta}}}(\BS{x})+u(\BS{x},\BS{\Tilde{\beta}})\big(\omega \Tilde{u}_{0}\big)^{-1}
\end{align*}
 Since $\BS{f}_{\BS{\Tilde{\beta}}}(\BS{0})=\BS{0}$ we have $\psi(\BS{0},\xi^*)=\Tilde{u}_{0}\big(\omega \Tilde{u}_{0}\big)^{-1}$ and according to Remark \ref{rem2.3.1.} $\xi^*$ is locally D-optimal if $ \Tilde{u}_{0}\big(\omega \Tilde{u}_{0}\big)^{-1}=\nu+1$ which holds true  if  $\omega=(\nu+1)^{-1}$.
\end{proof}

\begin{theorem} \label{theo3.3.1.}
Consider design $\xi^*\in \Xi_{0}$ for model $\mathcal{M}$. Let the design $\xi^*_{-\BS{0}}$ on $\mathcal{\widetilde{X}}$  be the conditional measure of $\xi^*$ given $\BS{x}\neq \BS{0}$.  Let a parameter point $\BS{\beta}=(\beta_0,\BS{\Tilde{\beta}}^\trp)^\trp$ be given such that $u(\BS{x},\BS{\beta})=\Tilde{u}(\BS{x},\BS{\Tilde{\beta}})$ for all $\BS{x}\in \mathcal{\widetilde{X}}$. Assume that $\mathcal{\widetilde{X}}\subseteq \mathcal{X}$ and  $\BS{f}_{\BS{\Tilde{\beta}}}(\BS{0})=\BS{0}$. Let ${\xi^*=(1/(\nu+1))\,\xi_{\BS{0}}+(\nu/(\nu+1))\,\xi^*_{-\BS{0}}}$.  Then \\
(1) If $\xi^*$ is locally D-optimal  (at $\BS{\beta}$) for model $\mathcal{M}$ then $\xi^*_{-\BS{0}}$ is locally D-optimal  (at $\BS{\Tilde{\beta}}$) for model $\widetilde{\mathcal{M}}$ .\\
(2) If  $\xi^*_{-\BS{0}}$ is  locally D-optimal (at $\BS{\Tilde{\beta}}$) for model $\widetilde{\mathcal{M}}$ and  
\begin{align}
\BS{f}^{\trp}_{\BS{\Tilde{\beta}}}(\BS{x})\BS{\Tilde{M}}^{-1}(\xi^*_{-\BS{0}},\BS{\Tilde{\beta}})\BS{f}_{\BS{\Tilde{\beta}}}(\BS{x})\le \nu\Big( 1-\frac{(\BS{c}^\trp\BS{f}_{\BS{\Tilde{\beta}}}(\BS{x})- \Tilde{u}^{\frac{1}{2}}(\BS{x},\BS{\Tilde{\beta}}))^2}{\Tilde{u}_{0}}\Big)\,\,\forall \BS{x}\in \mathcal{X} \label{eq3.16}
  \end{align}
  then $\xi^*$ is locally D-optimal  (at $\BS{\beta}$) for model $\mathcal{M}$.
\end{theorem}

\begin{proof}
\underbar{Ad ($1$)}\,\,Let $\xi^*=(1/(\nu+1))\,\xi_{\BS{0}}+(\nu/(\nu+1))\,\xi^*_{-\BS{0}} \in \Xi_{0}$  be  locally D-optimal  (at $\BS{\beta}$) on $\mathcal{X}$ for model $\mathcal{M}$. We want to proof that  $\xi^*_{-\BS{0}}$ on $\mathcal{\widetilde{X}}$ is locally D-optimal (at $\BS{\Tilde{\beta}}$) for model $\widetilde{\mathcal{M}}$. By condition (\ref{Equiv-D}) of The Equivalence Theorem  we guarantee at $\BS{\beta}=(\beta_0,\BS{\Tilde{\beta}}^\trp)^\trp$ that 
 \begin{equation}
u(\BS{x},\BS{\beta})\big(1,\BS{f}^\trp(\BS{x})\big)\BS{M}^{-1}(\xi^*,\BS{\beta})\big(1,\BS{f}^\trp(\BS{x})\big)^\trp \le \nu+1 \,\,\forall \BS{x}\in \mathcal{X},  \label{eq3.17}
\end{equation}
where, at $\BS{\beta}=(\beta_0,\BS{\Tilde{\beta}}^\trp)^\trp$, ${u}(\BS{x},\BS{\beta})= \Tilde{u}(\BS{x},\BS{\Tilde{\beta}})$ and $\BS{f}_{\BS{\beta}}(\BS{x})=\BS{f}_{\BS{\Tilde{\beta}}}(\BS{x})$ for all $\BS{x}\in \widetilde{\mathcal{X}}$ with $\widetilde{\mathcal{X}}\subseteq \mathcal{X}$. So $\BS{M}^{-1}\big(\xi^*,\BS{\beta}\big)=\BS{M}^{-1}\big(\xi^*,\BS{\Tilde{\beta}}\big)$ which is given by (\ref{eq3.15}) with  $\omega=1/(\nu+1)$.  Then inequality (\ref{eq3.17}) is equivalent to 
\begin{align*}
&\BS{f}^{\trp}_{\BS{\Tilde{\beta}}}(\BS{x})\Big( \frac{\nu+1}{\nu}\,\BS{\Tilde{M}}^{-1}(\xi^*_{-\BS{0}},\BS{\Tilde{\beta}})+\frac{(\nu+1)\BS{c}\BS{c}^\trp}{\Tilde{u}_{0}}\Big)\BS{f}_{\BS{\Tilde{\beta}}}(\BS{x})\nonumber \\
&\hspace{25ex} -\frac{2(\nu+1)\BS{c}^\trp\BS{f}_{\BS{\Tilde{\beta}}}(\BS{x})+(\nu+1)\Tilde{u}(\BS{x},\BS{\Tilde{\beta}})}{\Tilde{u}_{0}}\le \nu+1 \,\,\,\,\forall \BS{x}\in \widetilde{\mathcal{X}}.
\end{align*}
Elementary computations show that  the above inequality is equivalent to
\begin{align}
&\BS{f}^{\trp}_{\BS{\Tilde{\beta}}}(\BS{x}) \BS{\Tilde{M}}^{-1}(\xi^*_{-\BS{0}},\BS{\Tilde{\beta}})\BS{f}_{\BS{\Tilde{\beta}}}(\BS{x}) +\frac{\nu\,(\BS{c}^\trp\BS{f}_{\BS{\Tilde{\beta}}}(\BS{x})-\Tilde{u}^{\frac{1}{2}}(\BS{x},\BS{\Tilde{\beta}}))^2}{\Tilde{u}_{0}}\le \nu  \,\,\,\,\forall \BS{x}\in \widetilde{\mathcal{X}}.\label{eq3.18}\\
&\mbox{ Since }  \frac{\nu\,(\BS{c}^\trp\BS{f}_{\BS{\Tilde{\beta}}}(\BS{x})-\Tilde{u}^{\frac{1}{2}}(\BS{x},\BS{\Tilde{\beta}}))^2}{\Tilde{u}_{0}}\ge0, \mbox{ (\ref{eq3.18}) is equivalent to }\nonumber\\
&\BS{f}^{\trp}_{\BS{\Tilde{\beta}}}(\BS{x}) \BS{\Tilde{M}}^{-1}(\xi^*_{-\BS{0}},\BS{\Tilde{\beta}})\BS{f}_{\BS{\Tilde{\beta}}}(\BS{x})\le  \nu   \,\,\,\,\forall \BS{x}\in \widetilde{\mathcal{X}}. \nonumber
\end{align}
and so $\xi^*_{-\BS{0}}$ is locally D-optimal (at $\BS{\Tilde{\beta}}$) by condition (\ref{Equiv-D}) of The Equivalence Theorem.\\
\underbar{Ad ($2$)} Let $\xi^*_{-\BS{0}}$ on $\widetilde{\mathcal{X}}$ is locally D-optimal (at $\BS{\Tilde{\beta}}$) for model $\widetilde{\mathcal{M}}$. Under the assumptions stated in the theorem,  to show that $\xi^*$ from $\Xi_0$ on $\mathcal{X}$ is locally D-optimal  (at $\BS{\beta}$) for model $\mathcal{M}$ we investigate condition (\ref{Equiv-D}) of The Equivalence Theorem  which is  given above by  
(\ref{eq3.17}) and is also equivalent to (\ref{eq3.18}) at $\BS{\beta}$. Hence, (\ref{eq3.18})  holds true by condition ($\ref{eq3.16}$). Of course, because $ \xi^*_{-\BS{0}}$ is locally D-optimal  inequality ($\ref{eq3.16}$) becomes an equality at each design point of $\xi^*_{-\BS{0}}$ which surely is a design point of $\xi^*$ and since $\omega=1/(\nu+1)$ the equality also holds at the origin point $\BS{0}$.
\end{proof}

Next we introduce analogous result for the A-optimality. As ${\rm tr}\big(\BS{c}\BS{c}^\trp\big)=\BS{c}^\trp\BS{c}$ we obtain from (\ref{eq3.15}) 
\begin{align}
{\rm tr}\bigl(\BS{M}^{-1}(\xi,\BS{\Tilde{\beta}})\bigr)=\frac{1}{\Tilde{u}_{0}}\Bigg(\sqrt{\BS{c}^\trp\BS{c}+1} +\sqrt{{\rm tr}\bigl(\BS{\Tilde{M}}^{-1}(\xi_{-\BS{0}},\BS{\Tilde{\beta}})\bigr)}\Bigg)^2. \label{eq3.20}
\end{align}
Also from (\ref{eq3.15}) we get 
\begin{equation}
\BS{M}^{-2}(\xi,\BS{\Tilde{\beta}})=\left(\begin{array}{cc} \frac{\BS{c}^\trp\BS{c}+1}{\omega^2\,\Tilde{u}_{0}^2}  & -\frac{(\BS{c}^\trp\BS{c}+1)\BS{c}^\trp}{\omega^2\,\Tilde{u}_{0}^2}-\frac{\BS{c}^\trp\BS{\Tilde{M}}^{-1}(\xi_{-\BS{0}},\BS{\Tilde{\beta}})}{(1-\omega)\omega\Tilde{u}_{0}}\\ [0.3cm]   -\frac{\BS{c}(\BS{c}^\trp\BS{c}+1)}{\omega^2\,\Tilde{u}_{0}^2}-\frac{\BS{\Tilde{M}}^{-1}(\xi_{-\BS{0}},\BS{\Tilde{\beta}})\BS{c}}{(1-\omega)\omega\Tilde{u}_{0}}& \frac{(\BS{c}^\trp\BS{c}+1)\BS{c}\BS{c}^\trp}{\omega^2\,\Tilde{u}_{0}^2}+\frac{2\BS{\Tilde{M}}^{-1}(\xi_{-\BS{0}},\BS{\Tilde{\beta}})\BS{c}\BS{c}^\trp}{(1-\omega)\omega\Tilde{u}_{0}}+\frac{\BS{\Tilde{M}}^{-2}(\xi_{-\BS{0}},\BS{\Tilde{\beta}})}{(1-\omega)^2}
\end{array} \right). \label{eq3-inv}
\end{equation}

\begin{lemma} \label{lemm3.3-2}
Consider design  $\xi^*\in \Xi_{0}$ for model $\mathcal{M}$.  Let a parameter point $\BS{\beta}=(\beta_0,\BS{\Tilde{\beta}}^\trp)^\trp$ be given such that $u(\BS{x},\BS{\beta})=\Tilde{u}(\BS{x},\BS{\Tilde{\beta}})$ for all $\BS{x}\in \mathcal{\widetilde{X}}\subseteq \mathcal{X}$ and $\BS{f}_{\BS{\Tilde{\beta}}}(\BS{0})=\BS{0}$. Denote  ${\widetilde{\tau}={\rm tr}\bigl(\BS{\Tilde{M}}^{-1}(\xi_{-\BS{0}}^*,\BS{\Tilde{\beta}})\bigr)}$.  Then the design $\xi^*$ is  locally A-optimal (at $\BS{\beta}$) if it assigns weight $\omega$, below, to the origin $\BS{0}$;
 \[
\omega=\frac{\sqrt{\BS{c}^\trp\BS{c}+1}}{\sqrt{\BS{c}^\trp\BS{c}+1}+\sqrt{\Tilde{u}_{0}\,\widetilde{\tau}}}.
\]
\end{lemma}
\begin{proof} 
Under the assumptions given in the lemma we obtain $\BS{M}^{-2}(\xi^*,\BS{\Tilde{\beta}})$ from (\ref{eq3-inv}). Then  the sensitivity function obtained from condition (\ref{Equiv-A}) of The Equivalence Theorem  is given by 
\begin{align*}
\psi(\BS{x},\xi^*)&=u(\BS{x},\BS{\Tilde{\beta}})\big(1,\BS{f}^\trp(\BS{x})\big)\BS{M}^{-2}(\xi^*,\BS{\Tilde{\beta}})\big(1,\BS{f}^\trp(\BS{x})\big)^\trp\\
&=u(\BS{x},\BS{\Tilde{\beta}})\Big(\BS{f}^\trp(\BS{x})\Big(  \frac{(\BS{c}^\trp\BS{c}+1)\BS{c}\BS{c}^\trp}{\omega^2\,\Tilde{u}_{0}^2}+\frac{2\BS{\Tilde{M}}^{-1}(\xi_{-\BS{0}}^*,\BS{\Tilde{\beta}})\BS{c}\BS{c}^\trp}{(1-\omega)\omega\Tilde{u}_{0}}+\frac{\BS{\Tilde{M}}^{-2}(\xi_{-\BS{0}}^*,\BS{\Tilde{\beta}})}{(1-\omega)^2}\Big)\BS{f}(\BS{x})\\
&-2\Big(\frac{(\BS{c}^\trp\BS{c}+1)\BS{c}^\trp}{\omega^2\,\Tilde{u}_{0}^2}-\frac{\BS{c}^\trp\BS{\Tilde{M}}^{-1}(\xi_{-\BS{0}}^*,\BS{\Tilde{\beta}})}{(1-\omega)\omega\Tilde{u}_{0}}\Big)\BS{f}(\BS{x})+(\BS{c}^\trp\BS{c}+1)\big(\omega \Tilde{u}_{0}\big)^{-2}\Big)\\
&=\BS{f}_{\BS{\Tilde{\beta}}}^\trp(\BS{x})\Big(  \frac{(\BS{c}^\trp\BS{c}+1)\BS{c}\BS{c}^\trp}{\omega^2\,\Tilde{u}_{0}^2}+\frac{2\BS{\Tilde{M}}^{-1}(\xi_{-\BS{0}}^*,\BS{\Tilde{\beta}})\BS{c}\BS{c}^\trp}{(1-\omega)\omega\Tilde{u}_{0}}+\frac{\BS{\Tilde{M}}^{-2}(\xi_{-\BS{0}}^*,\BS{\Tilde{\beta}})}{(1-\omega)^2}\Big)\BS{f}_{\BS{\Tilde{\beta}}}(\BS{x})\\
&-2\Big(\frac{(\BS{c}^\trp\BS{c}+1)\BS{c}^\trp}{\omega^2\,\Tilde{u}_{0}^2}-\frac{\BS{c}^\trp\BS{\Tilde{M}}^{-1}(\xi_{-\BS{0}}^*,\BS{\Tilde{\beta}})}{(1-\omega)\omega\Tilde{u}_{0}}\Big)\BS{f}_{\BS{\Tilde{\beta}}}(\BS{x})+u(\BS{x},\BS{\Tilde{\beta}})(\BS{c}^\trp\BS{c}+1)\big(\omega \Tilde{u}_{0}\big)^{-2}.
\end{align*}
Since $\BS{f}_{\BS{\Tilde{\beta}}}(\BS{0})=\BS{0}$ we have $\psi(\BS{0},\xi^*)=\Tilde{u}_{0}(\BS{c}^\trp\BS{c}+1)\big(\omega \Tilde{u}_{0}\big)^{-2}$ and according to Remark \ref{rem2.3.1.} $\xi^*$ is locally A-optimal if $\Tilde{u}_{0}(\BS{c}^\trp\BS{c}+1)\big(\omega \Tilde{u}_{0}\big)^{-2}={\rm tr}(\BS{M}^{-1}(\xi^*,\BS{\Tilde{\beta}}))$ which holds true  if  ${\omega=\sqrt{  \frac{(\BS{c}^\trp\BS{c}+1)}{\Tilde{u}_{0}{\rm tr}(\BS{M}^{-1}(\xi^*,\BS{\Tilde{\beta}})}}}$. By (\ref{eq3.20}) we get $\omega=\frac{\sqrt{\BS{c}^\trp\BS{c}+1}}{\sqrt{\BS{c}^\trp\BS{c}+1}+\sqrt{\Tilde{u}_{0}\,\widetilde{\tau}}}.$
\end{proof}

\begin{theorem} \label{theo3.3.2.} Consider the assumptions and notations of Theorem \ref{theo3.3.1.} with ${\widetilde{\tau}={\rm tr}\bigl(\BS{\Tilde{M}}^{-1}(\xi_{-\BS{0}}^*,\BS{\Tilde{\beta}})\bigr)}$. Let
\[
\xi^*=\Bigg(\frac{\sqrt{\BS{c}^\trp\BS{c}+1}}{\sqrt{\BS{c}^\trp\BS{c}+1}+\sqrt{\Tilde{u}_{0}\widetilde{\tau}}}\Bigg)\,\xi_{\BS{0}}+\Bigg(\frac{\sqrt{\Tilde{u}_{0}\,\widetilde{\tau}}}{\sqrt{\BS{c}^\trp\BS{c}+1}+\sqrt{\Tilde{u}_{0}\,\widetilde{\tau}}}\Bigg)\,\xi^*_{-\BS{0}}.
\]
Denote the following equations 
\begin{align*}
T_{1}(\BS{x},\BS{\Tilde{\beta}})&=\frac{(\sqrt{\BS{c}^\trp\BS{c}+1}+\sqrt{\Tilde{u}_{0}\widetilde{\tau}})^2(\BS{c}^\trp\BS{f}_{\BS{\Tilde{\beta}}}(\BS{x})-\Tilde{u}^{\frac{1}{2}}(\BS{x},\BS{\Tilde{\beta}}))^2}{\Tilde{u}_{0}^2}\\
&+\frac{2(\sqrt{\BS{c}^\trp\BS{c}+1}+\sqrt{\Tilde{u}_{0}\widetilde{\tau}})^2}{\Tilde{u}_{0}\,\sqrt{\widetilde{\tau}\Tilde{u}_{0}(\BS{c}^\trp\BS{c}+1)}}\Bigg(\BS{f}^{\trp}_{\BS{\Tilde{\beta}}}(\BS{x}) \BS{\Tilde{M}}^{-1}(\xi^*_{-\BS{0}},\BS{\Tilde{\beta}})\BS{c}\BS{c}^\trp\BS{f}_{\BS{\Tilde{\beta}}}(\BS{x})\\
&-2\BS{c}^\trp \BS{\Tilde{M}}^{-1}(\xi^*_{-\BS{0}},\BS{\Tilde{\beta}})\Tilde{u}^{\frac{1}{2}}(\BS{x},\BS{\Tilde{\beta}})\BS{f}_{\BS{\Tilde{\beta}}}(\BS{x})\Bigg),\\
T_{2}(\BS{x},\BS{\Tilde{\beta}})&=2\sqrt{\frac{\widetilde{\tau} }{ \Tilde{u}_{0} (\BS{c}^\trp\BS{c}+1) } }\Bigg( \BS{f}^{\trp}_{\BS{\Tilde{\beta}}}(\BS{x}) \BS{\Tilde{M}}^{-1}(\xi^*_{-\BS{0}},\BS{\Tilde{\beta}})\BS{c}\BS{c}^\trp\BS{f}_{\BS{\Tilde{\beta}}}(\BS{x}) \\
&-\BS{c}^\trp \BS{\Tilde{M}}^{-1}(\xi^*_{-\BS{0}},\BS{\Tilde{\beta}})\Tilde{u}^{\frac{1}{2}}(\BS{x},\BS{\Tilde{\beta}})\BS{f}_{\BS{\Tilde{\beta}}}(\BS{x}) \Bigg).
\end{align*}
 Then \\
(1) If $\xi^*$ is locally A-optimal  (at $\BS{\beta}$) for model $\mathcal{M}$ and  $T_1(\BS{x},\BS{\Tilde{\beta}})\ge 0$ for all $\BS{x}\in \mathcal{\widetilde{X}}$ then $\xi^*_{-\BS{0}}$ is locally A-optimal  (at $\BS{\Tilde{\beta}}$) for model $\widetilde{\mathcal{M}}$.\\
(2) If  $\xi^*_{-\BS{0}}$ is  locally A-optimal (at $\BS{\Tilde{\beta}}$) for model $\widetilde{\mathcal{M}}$ and  
\begin{align}
\BS{\Tilde{f}}^{\trp}_{\BS{\Tilde{\beta}}}(\BS{x}) \BS{\Tilde{M}}^{-2}(\xi^*_{-\BS{0}},\BS{\Tilde{\beta}})\BS{\Tilde{f}}_{\BS{\Tilde{\beta}}}(\BS{x}) \le  \widetilde{\tau}\,\Big(1-\frac{(\BS{c}^\trp\BS{f}^{\trp}_{\BS{\Tilde{\beta}}}(\BS{x})-\Tilde{u}^{\frac{1}{2}}(\BS{x},\BS{\Tilde{\beta}}))^2}{\Tilde{u}_{0}}\Big)+T_{2}(\BS{x},\BS{\Tilde{\beta}}) \,\,\forall \BS{x}\in \mathcal{X}\label{eq3.21} 
\end{align}
then $\xi^*$ is locally A-optimal  (at $\BS{\beta}$) for model $\mathcal{M}$.
  \end{theorem}
 
\begin{proof}
\underbar{Ad ($1$)}\,\,Let $\xi^*=(\frac{\sqrt{\BS{c}^\trp\BS{c}+1}}{\sqrt{\BS{c}^\trp\BS{c}+1}+\sqrt{\Tilde{u}_{0}\,\widetilde{\tau}}})\,\xi_{\BS{0}}+(\frac{\sqrt{\Tilde{u}_{0}\,\widetilde{\tau}}}{\sqrt{\BS{c}^\trp\BS{c}+1}+\sqrt{\Tilde{u}_{0}\,\widetilde{\tau}}})\,\xi^*_{-\BS{0}} \in \Xi_0$ on $\mathcal{X}$ be  locally A-optimal  (at $\BS{\beta}$) for model $\mathcal{M}$.  We want to proof that  $\xi^*_{-\BS{0}}$ on $\mathcal{\widetilde{X}}$ is locally A-optimal (at $\BS{\Tilde{\beta}}$) for model $\widetilde{\mathcal{M}}$. Then condition (\ref{Equiv-A}) of  The Equivalence Theorem  guarantees at $\BS{\beta}=(\beta_0,\BS{\Tilde{\beta}}^\trp)^\trp$ that for all $\BS{x} \in \mathcal{X}$
 \begin{equation}
u(\BS{x},\BS{\beta})\big(1,\BS{f}^\trp(\BS{x})\big)\BS{M}^{-2}(\xi^*,\BS{\beta})\big(1,\BS{f}^\trp(\BS{x})\big)^\trp\leq {\rm tr}\bigl(\BS{\Tilde{M}}^{-1}(\xi ^*,\BS{\Tilde{\beta}})\bigr), \label{eq3.22}
\end{equation}
 where, at $\BS{\beta}=(\beta_0,\BS{\Tilde{\beta}}^\trp)^\trp$, ${u}(\BS{x},\BS{\beta})= \Tilde{u}(\BS{x},\BS{\Tilde{\beta}})$ and $\BS{f}_{\BS{\beta}}(\BS{x})=\BS{f}_{\BS{\Tilde{\beta}}}(\BS{x})$ for all $\BS{x}\in \widetilde{\mathcal{X}}$ with $\widetilde{\mathcal{X}}\subseteq \mathcal{X}$. So $\BS{M}^{-2}\big(\xi^*,\BS{\beta}\big)=\BS{M}^{-2}\big(\xi^*,\BS{\Tilde{\beta}}\big)$ which is given by (\ref{eq3-inv}) with  $\omega=(\sqrt{\BS{c}^\trp\BS{c}+1})/(\sqrt{\BS{c}^\trp\BS{c}+1}+\sqrt{\Tilde{u}_{0}\,\widetilde{\tau}})$. Then the l.h.s. of inequality (\ref{eq3.22}) equals 
\begin{align*}
&\BS{f}^{\trp}_{\BS{\Tilde{\beta}}}(\BS{x}) \Bigg(\frac{(\BS{c}^\trp\BS{c}+1)\BS{c}\BS{c}^\trp}{\omega^2\Tilde{u}_{0}^2}+\frac{2\BS{\Tilde{M}}^{-1}(\xi^*_{-\BS{0}},\BS{\Tilde{\beta}})\BS{c}\BS{c}^\trp}{\omega(1-\omega)\Tilde{u}_{0}}+\frac{1}{(1-\omega)^2}\BS{\Tilde{M}}^{-2}(\xi^*_{-\BS{0}},\BS{\Tilde{\beta}})\Bigg)\BS{f}_{\BS{\Tilde{\beta}}}(\BS{x}) \\
&-2\,\Bigg(\frac{\BS{c}^\trp(\BS{c}^\trp\BS{c}+1)}{\omega^2\,\Tilde{u}_{0}^2}+\frac{\BS{c}^\trp \BS{\Tilde{M}}^{-1}(\xi^*_{-\BS{0}},\BS{\Tilde{\beta}})}{\omega(1-\omega)\Tilde{u}_{0}}\Bigg)\Tilde{u}^{\frac{1}{2}}(\BS{x},\BS{\Tilde{\beta}})\BS{f}_{\BS{\Tilde{\beta}}}(\BS{x})+\frac{(\BS{c}^\trp\BS{c}+1)u(\BS{x},\BS{\Tilde{\beta}})}{\omega^2\,\Tilde{u}_{0}^2},
\end{align*}
and together with (\ref{eq3.20}) it is straightforward to see that  (\ref{eq3.22}) is equivalent to
\begin{equation}
\BS{f}^{\trp}_{\BS{\Tilde{\beta}}}(\BS{x}) \BS{\Tilde{M}}^{-2}(\xi^*_{-\BS{0}},\BS{\Tilde{\beta}})\BS{f}_{\BS{\Tilde{\beta}}}(\BS{x})+T_1(\BS{x},\BS{\Tilde{\beta}})\le \widetilde{\tau} \,\,\,\,\forall \BS{x}\in \widetilde{\mathcal{X}}.  \label{eq3.100}
\end{equation}
Since $T_1(\BS{x},\BS{\Tilde{\beta}})\ge 0$   for all   $\BS{x}\in \mathcal{\widetilde{X}}$,  (\ref{eq3.100}) is equivalent to 
\begin{equation*}
 \BS{f}^{\trp}_{\BS{\Tilde{\beta}}}(\BS{x}) \BS{\Tilde{M}}^{-2}(\xi^*_{-\BS{0}},\BS{\Tilde{\beta}})\BS{f}_{\BS{\Tilde{\beta}}}(\BS{x}) \le  \widetilde{\tau}  \,\,\,\,\forall \BS{x}\in \widetilde{\mathcal{X}}.
\end{equation*}
and so $\xi^*_{-\BS{0}}$ is locally A-optimal (at $\BS{\Tilde{\beta}}$) by condition (\ref{Equiv-A}) of The Equivalence Theorem.\\
\underbar{Ad ($2$)} Let $\xi^*_{-\BS{0}}$ on $\widetilde{\mathcal{X}}$ is locally A-optimal (at $\BS{\Tilde{\beta}}$) for model $\widetilde{\mathcal{M}}$. Under the assumptions stated in the theorem to show that $\xi^*$ from $\Xi_0$ on $\mathcal{X}$ is locally A-optimal  (at $\BS{\beta}$) for model $\mathcal{M}$  we investigate condition (\ref{Equiv-A}) of The Equivalence Theorem  which is given above by  
(\ref{eq3.22}) and  is also equivalent to (\ref{eq3.100}) at $\BS{\beta}$ for all $\BS{x}\in \mathcal{X}$. Hence, it is straightforward to see that  (\ref{eq3.100}) for all $\BS{x}\in \mathcal{X}$ holds true by condition (\ref{eq3.21}). Of course, because $\xi^*_{-\BS{0}}$ is locally A-optimal and $T_{2}(\BS{x},\BS{\Tilde{\beta}})=0$ for all $\BS{x}\in \mathrm{supp}(\xi^*_{-\BS{0}})$   inequality (\ref{eq3.21}) becomes an equality at each design point of $\xi^*_{-\BS{0}}$ which surely is a design point of $\xi^*$. Since  $\omega=(\sqrt{\BS{c}^\trp\BS{c}+1})/(\sqrt{\BS{c}^\trp\BS{c}+1}+\sqrt{\Tilde{u}_{0}\,\widetilde{\tau}})$ and $T_{2}(\BS{0},\BS{\Tilde{\beta}})=0$  the equality also holds at the origin point $\BS{0}$.
\end{proof}
\begin{remark} The results of this section  might be viewed as a generalization of the results of  both \citet{Li2005} and \citet{ZHANG2013196} that were derived under linear models, i.e.,  when the intensities are constants equal to 1. 
\end{remark}
 
\begin{remark}
A design with minimal support, i.e., the support size  equals the dimension of $\BS{f}$ ($r = p$) is called a saturated design. In fact, the assumption $\BS{c}^\trp\BS{f}(\BS{x})=1$ for all $\BS{x}\in \mathrm{supp}(\xi^*)\setminus\{\BS{0}\}$  is equivalent to that  $\BS{f}(\BS{x})$ for all $\BS{x}\in \mathrm{supp}( \xi^*_{-\BS{0}})$ lies  on a hyperplane. Thus every saturated design for  generalized linear models without intercept satisfies that assumption.  Moreover, the assumption $\BS{c}^\trp\BS{f}(\BS{x})=1$ for all $\BS{x}\in\widetilde{\mathcal{X}}$ is satisfied when $\widetilde{\mathcal{X}}$ is given by the $(\nu-1)$-dimensional unit simplex, i.e., $\widetilde{\mathcal{X}}=\{\BS{x}=(x_1,\dots,x_\nu)^\trp, 0\le x_i\le1\,\, \forall i, \sum_{i=1}^\nu x_i=1\}$. In such a case the mixture constraint of $\widetilde{\mathcal{X}}$ which is given by $\sum_{i=1}^{\nu}x_i=1$  entails that $\BS{c}=(1,\dots,1)^\trp$. 
\end{remark}

\section{Applications}

\subsection{ Poisson models}

We consider a first order Poisson  model with $\BS{f}(\BS{x})=(1,\BS{x}^\trp)^\trp$. The intensity functions under $\mathcal{M}$ and $\widetilde{\mathcal{M}}$ are given by 
\[
u(\BS{x},\BS{\beta})=\exp(\beta_0+\BS{x}^\trp\BS{\Tilde{\beta}})\,\, \mbox{ and }\,\, \Tilde{u}(\BS{x},\BS{\Tilde{\beta}})=\exp(\BS{x}^\trp\BS{\Tilde{\beta}}), 
\]
 respectively. It is noted that $u(\BS{x},\BS{\beta})$ factorizes; i.e., $u(\BS{x},\BS{\beta}) =\exp(\beta_0)\Tilde{u}(\BS{x},\BS{\Tilde{\beta}})$. Therefore, $\BS{M}(\xi,\BS{\beta})=\exp(\beta_0)\BS{M}(\xi,\BS{\Tilde{\beta}})$ for any given parameter point $\BS{\beta}=(\beta_0,\BS{\Tilde{\beta}}^\trp)^\trp$.  That means the design $\xi$ is independent of $\beta_0$ and hence, locally optimal designs for a Poisson model with intercept is governed by $\Tilde{u}(\BS{x},\BS{\Tilde{\beta}})$.  Similar situation holds under  the Rasch Poisson-Gamma counts model (\citet{10.1007/978-3-319-00218-7_14}) in item response theory and the  Rasch Poisson counts model (\citet{2018arXiv181003893G}).  \par

A relevant work from the literature  includes the results of \citet{10.2307/24308852} who  derived a  locally D-optimal saturated design $\xi^*$ for a first order Poisson model with intercept on $\mathcal{X}= [0,1]^\nu$ where \,$\nu\ge 2$ at $\beta_i=-2 \,\,(1\le i\le \nu)$. The support is given by $\BS{x}_0^*=(0,0,\dots,0)^\trp$ and  the $\nu$-dimensional unit vectors $\BS{x}_i^*=\BS{e}_i \,\,(1\le i \le \nu)$ with equal weights $(\nu+1)^{-1}$. So under the assumptions of Theorem  \ref{theo3.3.1.}, part (1) with $\BS{c}=\BS{1}_{\nu}$ as the $\nu$-vector of ones, the design $\xi^*_{-\BS{0}}$ on $\mathcal{X}$ is locally D-optimal at $\beta_i=-2 \,\,(1\le i\le \nu)$ for the corresponding model without intercept.  

\subsection{ Logistic models}
Consider a first order logistic model with  $\BS{f}(\BS{x})=(1,\BS{x}^\trp)^\trp$.  The intensity functions under $\mathcal{M}$ and $\widetilde{\mathcal{M}}$ are given by 
\[
u(\BS{x},\BS{\beta})=\frac{\exp(\beta_0+\BS{x}^\trp\BS{\Tilde{\beta}})}{(1+\exp(\beta_0+\BS{x}^\trp\BS{\Tilde{\beta}}))^2}\,\,\, \mbox{    and   }\,\,\,   \Tilde{u}(\BS{x},\BS{\Tilde{\beta}})=\frac{\exp(\BS{x}^\trp\BS{\Tilde{\beta}})}{(1+\exp(\BS{x}^\trp\BS{\Tilde{\beta}}))^2}, 
\]
 respectively. Note that $u(\BS{x},\BS{\beta})=\Tilde{u}(\BS{x},\BS{\Tilde{\beta}})$ and $\BS{M}(\xi,\BS{\beta})=\BS{M}(\xi,\BS{\Tilde{\beta}})$ at  $\BS{\beta}=(0,\BS{\Tilde{\beta}}^\trp)^\trp$. \par

In the literature \citet{doi:10.1080/02331888.2014.937342}, Theorem 3.2, provided a three-point  locally D-optimal  saturated design $\xi^*$  at $(0,\BS{\Tilde{\beta}}^\trp)^\trp$, $\BS{\Tilde{\beta}}\in (0,\infty)^2 $  for the two-factor logistics model on the experimental region $\mathcal{X}=  [0,\infty)^2$.   The support is given by $(0,0)^\trp, (0,u^*)^\trp,(u^*,0)^\trp$ where $u^*> 0$ is the unique solution for $u$ to the equation $2+u+2e^u-ue^u=0$. Hence, the assumptions of Theorem  \ref{theo3.3.1.}, part (1) with $\BS{c}=(1/u^*, 1/u^*)^\trp$ are satisfied  and hence the design $\xi^*_{-\BS{0}}$ on $\mathcal{X}$ is locally D-optimal (at $\BS{\Tilde{\beta}}$) with equal weights $1/2$ for the corresponding model without intercept.\par
See also Example 3 in \citet{schmidt2017optimal} where product type designs are locally D-optimal at $\BS{\beta}=(0,\BS{\Tilde{\beta}}^\trp)^\trp$ for logistic models with intercept.

\section{Extensions} \label{sec5}

The obtained results in Section 3 under generalized linear models  might be applicable under another nonlinear models that are defined by
 \begin{equation}
Y(\BS{x})=h(\BS{x},\BS{\beta})+\varepsilon\,\,\mbox{ where }\,\,  \varepsilon \,\,\mbox{is the error term}. \label{eq2.6}
  \end{equation} 
In this context we define $\BS{f}_{\BS{\beta}}(\BS{x})$ to be the gradient vector of $h(\BS{x},\BS{\beta})$, i.e.,
\begin{equation}
\BS{f}_{\BS{\beta}}(\BS{x})=\nabla  h(\BS{x},\BS{\beta})=\frac{\partial h(\BS{x},\BS{\beta}) }{\partial \BS{\beta}}=\Big(\frac{\partial h(\BS{x},\BS{\beta}) }{\partial \beta_1},\dots,\frac{\partial h(\BS{x},\BS{\beta}) }{\partial \beta_p}\Big)^\trp.  \label{eq2.7}
\end{equation}
The Fisher information matrix at a point $\BS{x}\in \mathcal{X}$ is given by $\BS{M}(\BS{x},\BS{\beta})=\BS{f}_{\BS{\beta}}(\BS{x})\BS{f}_{\BS{\beta}}^\trp(\BS{x})$. Actually, nonlinear models of form (\ref{eq2.6}) were discussed carefully in the literature (see  \citet{doi:10.1080/00401706.1989.10488475}, \citet{ATKINSON1996437}). Here, generally, a nonlinear model includes explicitly an intercept term if the  function $\BS{f}_{\BS{\beta}}(\BS{x})$ includes the constant $1$ (see \citet {10.1007/978-3-662-12516-8_8},  \citet{LI2011644}, \citet{doi:10.1080/02331888.2014.922562}, \citet{he2018optimal}).  In \citet{doi:10.1198/016214508000000427} some dose–response nonlinear models with intercept were listed, e.g., 
\begin{align*}
&E_{\rm{max}} \mbox{ model }: h(\BS{x},\BS{\beta})=\beta_0+\frac{\beta_1 x}{x+\beta_2},\,\,\mbox{with}\,\,\BS{f}_{\BS{\beta}}(\BS{x})=\Big(1,\frac{x}{x+\beta_1},\frac{-\beta_1 x}{(x+\beta_1)^2}\Big)^\trp \\
&\mbox{Exponential model}:  h(\BS{x},\BS{\beta})=\beta_0+\beta_1\exp(\frac{x}{\beta_2}),\,\,\mbox{with}\,\,\BS{f}_{\BS{\beta}}(\BS{x})=\Big(1,\exp(\frac{x}{\beta_2}), -\frac{\beta_1 x \exp(\frac{x}{\beta_2})}{\beta_2^2}\Big)^\trp.
\end{align*}
The above nonlinear models were also considered in \citet{10.2307/25734102} and locally D-optimal designs on the experiential region $[0,150]$ were derived under zero intercept, i.e., $\beta_0=0$. The support is given by $\{0, x^*, 150\}$ with equal weights $1/3$ where $x^*\in (0,150)$ is obtained analytically. \par
 In analogy to the results derived under GLMs in Section 3 we denote $\BS{\beta}=(\beta_0,\BS{\Tilde{\beta}}^\trp)^\trp$ and we can write the Fisher information matrix of $\xi$ on $\mathcal{\widetilde{X}}$  under a non-intercept nonlinear model as ${\BS{\Tilde{M}}(\xi,\BS{\Tilde{\beta}})=\int_{\mathcal{\widetilde{X}}} \BS{f}_{\BS{\Tilde{\beta}}}(\BS{x}) \BS{f}_{\BS{\Tilde{\beta}}}^\trp(\BS{x})\, \xi(\mathrm{d}\BS{x})}$, while the Fisher information matrix of $\xi$ on $\mathcal{X}$ under a nonlinear model with intercept is $\BS{M}(\xi,\BS{\beta})=\int_{\mathcal{X}} \big(1,\BS{f}_{\BS{\beta}}^\trp(\BS{x})\big)^\trp \big(1,\BS{f}_{\BS{\beta}}^\trp(\BS{x})\big)\, \xi(\mathrm{d}\BS{x})$. The following results are immediate.

\begin{corollary} \label{cor3.1}
Let the design $\xi^*$ be defined on $\mathcal{X}$ such that $\BS{0}\in \mathrm{supp}(\xi^*)$. Let the design $\xi^*_{-\BS{0}}$ on $\mathcal{\widetilde{X}}$  be the  conditional measure of $\xi^*$ given $\BS{x}\neq \BS{0}$  such that $\mathcal{\widetilde{X}}\subseteq \mathcal{X}$. Given a parameter point $\BS{\beta}=(\beta_0,\BS{\Tilde{\beta}}^\trp)^\trp$  such that $\BS{f}_{\BS{\beta}}(\BS{x})=\BS{f}_{\BS{\Tilde{\beta}}}(\BS{x})$ for all $\BS{x}\in \mathcal{\widetilde{X}}$ with  $\BS{f}_{\BS{\Tilde{\beta}}}(\BS{0})=\BS{0}$. Then assume there exist a constant vector $\BS{c}$ such that $\BS{c}^\trp\BS{f}_{\BS{\Tilde{\beta}}}(\BS{x})=1$ for all $\BS{x}\in \mathrm{supp}(\xi^*)\setminus\{\BS{0}\}$.  Let $\xi^*=(1/(\nu+1))\,\xi_{\BS{0}}+(\nu/(\nu+1))\,\xi^*_{-\BS{0}}$. Then \\
(1) If $\xi^*$ is locally D-optimal  (at $\BS{\beta}$) for model with intercept then $\xi^*_{-\BS{0}}$ is locally D-optimal  (at $\BS{\Tilde{\beta}}$) for  the corresponding model without intercept.\\
(2) If  $\xi^*_{-\BS{0}}$ is  locally D-optimal (at $\BS{\Tilde{\beta}}$) for model without intercept and  
\begin{align}
\BS{f}^{\trp}_{\BS{\Tilde{\beta}}}(\BS{x})\BS{\Tilde{M}}^{-1}(\xi^*_{-\BS{0}},\BS{\Tilde{\beta}})\BS{f}_{\BS{\Tilde{\beta}}}(\BS{x})\le \nu\Big( 1-(\BS{c}^\trp\BS{f}_{\BS{\Tilde{\beta}}}(\BS{x})-1)^2\Big)\,\,\forall \BS{x}\in \mathcal{X} \label{eq4-6}
  \end{align}
  then $\xi^*$ is locally D-optimal  (at $\BS{\beta}$) for the corresponding model with intercept.
\end{corollary}

\begin{corollary} \label{cor3.2} Under assumptions and notations of Corollary \ref {cor3.1} with ${\widetilde{\tau}={\rm tr}\bigl(\BS{\Tilde{M}}^{-1}(\xi_{-\BS{0}}^*,\BS{\Tilde{\beta}})\bigr)}$.  
 Let 
\[
\xi^*=\Bigg(\frac{\sqrt{\BS{c}^\trp\BS{c}+1}}{\sqrt{\BS{c}^\trp\BS{c}+1}+\sqrt{\widetilde{\tau}}}\Bigg)\,\xi_{\BS{0}}+\Bigg(\frac{\sqrt{\widetilde{\tau}}}{\sqrt{\BS{c}^\trp\BS{c}+1}+\sqrt{\widetilde{\tau}}}\Bigg)\,\xi^*_{-\BS{0}}.\]
Denote the following equations 
\begin{align*}
T_{1}(\BS{x},\BS{\Tilde{\beta}})&=(\sqrt{\BS{c}^\trp\BS{c}+1}+\sqrt{\widetilde{\tau}})^2(\BS{c}^\trp\BS{f}_{\BS{\Tilde{\beta}}}(\BS{x})-1)^2\\
&+\frac{2(\sqrt{\BS{c}^\trp\BS{c}+1}+\sqrt{\widetilde{\tau}})^2}{\sqrt{\widetilde{\tau}(\BS{c}^\trp\BS{c}+1)}}\Bigg(\BS{f}^{\trp}_{\BS{\Tilde{\beta}}}(\BS{x}) \BS{\Tilde{M}}^{-1}(\xi^*_{-\BS{0}},\BS{\Tilde{\beta}})\BS{c}\BS{c}^\trp\BS{f}_{\BS{\Tilde{\beta}}}(\BS{x})\\
&-2\BS{c}^\trp \BS{\Tilde{M}}^{-1}(\xi^*_{-\BS{0}},\BS{\Tilde{\beta}})\BS{f}_{\BS{\Tilde{\beta}}}(\BS{x})\Bigg),\\
T_{2}(\BS{x},\BS{\Tilde{\beta}})&=2\sqrt{\frac{\widetilde{\tau} }{ \BS{c}^\trp\BS{c}+1 } }\Bigg( \BS{f}^{\trp}_{\BS{\Tilde{\beta}}}(\BS{x}) \BS{\Tilde{M}}^{-1}(\xi^*_{-\BS{0}},\BS{\Tilde{\beta}})\BS{c}\BS{c}^\trp\BS{f}_{\BS{\Tilde{\beta}}}(\BS{x}) \\
&-\BS{c}^\trp \BS{\Tilde{M}}^{-1}(\xi^*_{-\BS{0}},\BS{\Tilde{\beta}})\BS{f}_{\BS{\Tilde{\beta}}}(\BS{x}) \Bigg).
\end{align*}
 Then \\
(1) If $\xi^*$ is locally A-optimal  (at $\BS{\beta}$) for a model with intercept and  $T_1(\BS{x},\BS{\Tilde{\beta}})\ge 0$ for all $\BS{x}\in \mathcal{\widetilde{X}}$ then $\xi^*_{-\BS{0}}$ is locally A-optimal  (at $\BS{\Tilde{\beta}}$) for the corresponding model without intercept.\\
(2) If  $\xi^*_{-\BS{0}}$ is  locally A-optimal (at $\BS{\Tilde{\beta}}$) for a model without intercept and  
\begin{align*}
\BS{f}^{\trp}_{\BS{\Tilde{\beta}}}(\BS{x}) \BS{\Tilde{M}}^{-2}(\xi^*_{-\BS{0}},\BS{\Tilde{\beta}})\BS{f}_{\BS{\Tilde{\beta}}}(\BS{x}) \le  \widetilde{\tau}\,\Big(1-(\BS{c}^\trp\BS{f}^{\trp}_{\BS{\Tilde{\beta}}}(\BS{x})-1)^2\Big)+T_{2}(\BS{x},\BS{\Tilde{\beta}}) \,\,\forall \BS{x}\in \mathcal{X}
\end{align*}
then $\xi^*$ is locally A-optimal  (at $\BS{\beta}$) for the corresponding model with intercept.
  \end{corollary}

\begin{remark}
In view of  the assumptions of the previous corollaries $\BS{M}^{-1}(\xi,\BS{\Tilde{\beta}})$ is given by  (\ref{eq3.15}) where $\Tilde{u}_{0}$ vanishes. That is due to $\BS{c}^\trp\BS{\Tilde{m}}(\xi_{-\BS{0}},\BS{\Tilde{\beta}})=1$,    $\BS{\Tilde{M}}^{-1}(\xi_{-\BS{0}},\BS{\Tilde{\beta}})\BS{\Tilde{m}}(\xi_{-\BS{0}},\BS{\Tilde{\beta}})=\BS{c}$ thus  $\BS{\Tilde{m}}^\trp(\xi_{-\BS{0}},\BS{\Tilde{\beta}})  \BS{\Tilde{M}}^{-1}(\xi_{-\BS{0}},\BS{\Tilde{\beta}})\BS{\Tilde{m}}(\xi_{-\BS{0}},\BS{\Tilde{\beta}})=1$.
\end{remark}


\end{document}